\documentclass[11pt,two side]{article}

\usepackage[right=.95 in, left=.95 in]{geometry}

\usepackage{amsmath,amssymb, amsthm, bbm, mathtools, wasysym, graphicx, fancyhdr, color, authblk}

\usepackage{ulem,cancel}

\numberwithin{equation}{section}

\newtheorem{thm}{Theorem}[section]

\newtheorem{lem}[thm]{Lemma}

\newcommand{\eps}{\varepsilon}

\newcommand{\EEE}{{\mathbb{E}}}

\newcommand{\RRR}{{\mathbb{R}}}

\newcommand{\cC}{{\mathcal{C}}}

\newcommand{\cF}{{\mathcal{F}}}

\newcommand{\cL}{{\mathcal{L}}}

\newcommand{\cN}{{\mathcal{N}}}

\newcommand{\grg}{{\gamma}}
\newcommand{\grd}{{\delta}}

\newcommand{\grh}{{\eta}}
\newcommand{\grth}{{\theta}}

\newcommand{\grm}{{\mu}}
\newcommand{\grn}{{\nu}}

\newcommand{\grr}{{\rho}}
\newcommand{\grs}{{\sigma}}
\newcommand{\grt}{{\tau}}

\newcommand{\grD}{{\Delta}}

\newcommand{\1}{{\mathbbm{1}}}
\newcommand{\dd}{{\mathrm{d}}}

\newcommand{\del}{{\partial}}

\DeclareMathOperator{\sgn}{sgn}

\title{Exit time asymptotics for dynamical systems with fast random switching near an unstable equilibrium  }
\author{Yuri Bakhtin}
\author{Alexisz Ga\'al}
\affil{Courant Institute of Mathematical Sciences \footnote{251 Mercer St, New York, NY 10012}}

\date{\today}
\begin{document}
\maketitle

\abstract{We consider the exit problem for a one-dimensional system with random switching near an unstable equilibrium point of the averaged drift. In the infinite switching rate limit, we show that the exit time satisfies a limit theorem
with a logarithmic deterministic term and a random correction converging in distribution. Thus this setting is in the universality class of the unstable equilibrium exit under small white-noise perturbations. }



\section{Introduction and Main result}

In \cite{Kifer}, \cite{Eizenberg:MR749377}, \cite{Day:MR1376805},
\cite{Bakhtin-SPA:MR2411523}, \cite{Bak2010}, \cite{Bak2011}, \cite{AB2011},
\cite{Bakhtin-Correll:MR2983392}, \cite{Bakhtin:MR2935120}, \cite{PPG:doi:10.1142/S0219493719500229},
\cite{BPG-AAP}, \cite{BPG:equidistr},
exit problems for processes in neighborhoods of unstable equilibria under the influence of white noise
of small magnitude $\eps$ were studied. Among other results, it was obtained (under various additional sets of technical assumptions in \cite{Day:MR1376805}, \cite{Bakhtin-SPA:MR2411523}, \cite{Bak2011},    \cite{AB2011})
that if the origin is a hyperbolic critical point
of a smooth vector field $b$, with simple leading eigenvalue $\lambda>0$ of the linearization, then for any initial condition belonging to the stable manifold of the origin, the time $\tau_\eps$
when the solution of the It\^o SDE 
\begin{equation}
\label{eq:sde}
\dd X_\eps(t) = b(X_\eps(t))\dd t +\eps \sigma(X_\eps(t))\dd W(t),
\end{equation}
with nondegenerate smooth diffusion matrix $\sigma$ 
exits a small neighborhood $U$ of the origin satisfies the following limit theorem:
there are numbers $D_{\pm1}$ such that
\begin{equation}
\left(X_\eps(\tau_\eps),\ \tau_\eps -\frac{1}{\lambda}\ln\frac{1}{\eps}\right) \stackrel{d}{\longrightarrow} \left(q_{B},\  D_{B}+\frac{1}{\lambda}\ln|N|\right),\quad \eps\to 0,
\label{eq:limit for diffusion}
\end{equation}
where $q_{\pm1}$ are the points where the invariant curve associated with the eigenvalue $\lambda$
intersects the boundary of $U$, $B$ and $N$ are independent random variables, $B$ is $1/2$-Bernoulli, $N$ is standard Gaussian. 

It is clear that such an asymptotic result for solutions of an SDE must describe the asymptotic behavior for a whole class of systems that are well approximated by this kind of SDE or its exemplar one-dimensional additive noise linear version 
\[
\dd X_\eps(t)=\lambda X_\eps(t)\dd t+\eps \dd W.
\]

In fact, it was shown in~\cite{Bakhtin:MR2935120} that the exit times for Glauber dynamics for the Curie--Weiss mean field model belong to the universality class associated with~\eqref{eq:limit for diffusion}, i.e., they satisfy a similar limit theorem under the infinite system size limit, 
with limiting distribution being $\ln|N|$ up to scaling and translation. These results were used in \cite{Bakhtin-Correll:MR2983392} in the context of decision/reaction times in psychology.

It is interesting to explore this universality class further and study processes of totally different nature that exhibit similar behavior. In this paper, we study a family of processes with random switching, also known under the names of hybrid systems, piecewise deterministic Markov processes (PDMP), and random evolutions. 
The bibliography on these processes is growing. Interestingly, they were introduced and rediscovered many times by different groups of researchers.  Here we give just a few references to works of some of these groups:
\cite{Kac-telegraph:MR0510166},  \cite{Davis}, \cite{Hersh-story:MR1962927},  
\cite{Anisimov:MR2437051}, \cite{Faggionato}, \cite{Yin},  \cite{Malrieu},  \cite{Bakhtin-Hurth}, \cite{LawleyMattinglyReed2015}, \cite{Malrieu_2015}.

 In general, these processes are defined by a family of vector fields and a collection of rates of switching between those vector fields. 
At each time the system is in the state where it evolves along one of the vector fields from the family. At random times, the system jumps between states switching active vector fields from one to another according to the prescribed Markovian rates.

In the limit of infinite switching rates, the evolution can be effectively described by the law of large numbers through the averaging of the vector fields involved. One can also state a central limit theorem and, moreover, a functional central limit theorem for such systems on a finite time interval, see, e.g.,
\cite[Chapter 4]{Anisimov:MR2437051}.

More interesting questions involve the behavior of such systems over unbounded time intervals.
In this paper, we study a class of switching processes on time scales logarithmic in the switching rate $\mu$ and show that it belongs to the universality class associated with~\eqref{eq:limit for diffusion}. We consider processes driven intermittently by two vector fields  in one dimension. The
main condition on these vector fields is that their average defines an unstable critical point. The exit from a neighborhood of that unstable  equilibrium takes  a logarithmically long time in the switching rate~$\mu$. Thus, similarly to the situation in~\cite{Bakhtin:MR2935120},  the usual techniques of weak convergence on finite time horizon are not sufficient for obtaining the desired universality result and have to be supplemented with additional arguments.

Let us now describe the system we are interested in and our main result more formally.
Let $f_1,f_{-1}\in\cC^2(\RRR)$ with $a_1=f_1'(0)$ and $a_{-1}=f_{-1}'(0)$. We also 
define
\begin{equation}
\label{eq:F}
F(x)= \frac{1}{2}(f_1(x)+f_{-1}(x)),\quad x\in\RRR,
\end{equation}
and require that $a=F'(0)=(a_1+a_{-1})/2 >0$. Furthermore, we require that $|f_1(0)|=|f_{-1}(0)|>0$ and $\sgn (F(x))=\sgn(x)$, where $\sgn (z)= z/|z|$ for $z\not= 0$ and $\sgn (0)=0$.
We will study the dynamics driven by these functions on a finite segment $[-R,R]$ for some $R>0$, so
without loss of generality we will assume that the functions $f_{\pm1}$ and their first two derivatives
are bounded.

An example of such pair of functions is given by $f_1(x)=e^{ax}$, $f_{-1}(x)=-e^{-ax}$ for $x\in[-R,R]$ and $a>0$.

Let $(\grs^\mu_t)_{t\geq 0}$ be a homogeneous, rate $\mu >0$, right continuous Markov process on the state space $\{+1 ,-1\}$ with an arbitrary initial distribution on $\{+1 ,-1\}$ at time $0$. The realizations of this process almost surely make finitely many switches between $1$ and $-1$ on any finite time interval, so omitting
the exceptional set, we can work on a probability space $(\Omega,\cF,P)$ that guarantees that the
number of switches is locally finite for {\it all} realizations.

We will study the random trajectories $(x^\mu_t)_{t\geq 0}$ defined by 
\begin{align*}
\frac{\dd x^\mu_t}{\dd t}&=f_{\grs^\mu_t}(x^\mu_t),\quad t\ge 0, \\
x^\mu_0&=0.
\end{align*}
The paths of the stochastic process $(x^\mu_t)_{t\geq 0}$ are continuous. They switch between the dynamics governed by $f_{1}$ and $f_{-1}$ intermittently, being controlled by the Markov chain $(\grs^\mu_t)_{t\geq 0}$. For any $r>0$, we define the exit time from the interval $[-r,r]$:
\[
\grt^\mu(r)=\inf\{t: |x^\mu_t|\geq r\}.
\]

Our main result describes the joint asymptotic behavior of the exit time $\grt^\mu(r)$ and 
exit location $x^\mu_{\grt^\mu(r)}$ as $\mu \to \infty$. The sign of $x^\mu_{\grt^\mu(r)}$ can be interpreted as a decision between two alternative directions of exit made by the system by the time $\grt^\mu(r)$. Let us define
\begin{equation}
\label{eq:K}
K(r)=\int_0^r \left(\frac{1}{F(x)}-\frac{1}{ax}\right) \dd x, \quad r\not=0.
\end{equation}
Let us also define $D(0)$ to be arbitrary and
\[
D(r)=K(r)+\frac{\log |r|}{a}+\frac{\log (\sqrt{2a}/|f_1(0)|)}{a}, \quad r\not=0.
\]

\begin{thm}For any $r>0$, as $\mu\to\infty$,
\[
\left(x^\mu_{\grt^\mu(r)},\ \grt^\mu(r)- \frac{1}{2a}\log\mu  \right)\ \overset{d}{\longrightarrow}\  \left(r\cdot\sgn N,\ - \frac{1}{a}\log |N| + D(r\cdot \sgn N)\right),
\]
where $N$ is a standard Gaussian random variable.
\label{mainthm}
\end{thm}

Our strategy for the proof of Theorem \ref{mainthm} consists of studying the asymptotic exit time of $(x^\mu_t)_{t\geq 0}$ from the interval $[-\mu^\grg, \mu^\grg]$ for $\grg\in(1/4, 1/2)$, then the remaining time to hit the boundary $\{r, -r\}$. In Lemma \ref{exitth}, we show that up to an additive constant depending on $a$ and $|f_1(0)|$ (the last term in the definition of $D$), the time needed to exit $[-\mu^\grg, \mu^\grg]$ is $- \frac{1}{a}\log |N|+\frac{1/2-\grg}{a}\log \mu$, after which the process becomes deterministic in the limit $\mu\to\infty$, and is driven by $F$, see Lemma \ref{supr}. By Lemma \ref{dettime}, the process driven by $F$ and started from $\{\mu^\grg, -\mu^\grg\}$ requires $K(r\cdot \sgn N)+\frac{\log|r|}{a}+\frac{\grg}{a}\log\mu$ time to hit $\{r, -r\}$, depending on the sign of exit direction $\sgn N$. Summing up the two contributions to $\grt^\mu(r)$, the terms $\pm \frac{\grg}{a}\log\mu$ cancel, and we obtain Theorem \ref{mainthm}.

{\bf Acknowledgment.} YB is grateful to NSF for partial support via grant DMS-1811444.

\section{Proof}

For brevity, throughout this section, we will often omit $\mu$ in the notation and use $\grs_t=\grs^\mu_t$, $x_t=x^\mu_t$, $\tau(r)=\tau^\mu(r)$, etc.

By Taylor's theorem, there are functions $R_{1}, R_{-1}:\RRR\to\RRR$ such that 
\begin{equation}
f_\grs(x)=\grs f_1(0)+a_\grs x+R_\grs(x),\quad  x\in\RRR,\ \grs\in\{1,-1\},
\label{eq:Taylor}
\end{equation}
and $|R_\grs(x)|\leq c \,  2^{-1} x^2$, where
\begin{equation}
\label{eq:c}
c=\max_{\grs \in \{-1,1\}}\sup_{x\in\RRR}|f_\grs''(x)|.
\end{equation} 

The generator of the Markov process $(\grs_t,x_t)$ on any bounded, smooth function  $g:\{-1,1\}\times\RRR\to\RRR$ is given by
\begin{equation}
\cL g(\grs,x)=f_\grs(x) \del_x g(\grs,x) + \mu (g(-\grs,x)-g(\grs,x)).
\label{generator}
\end{equation}
Applying \eqref{generator} to the functions $g(\grs,x)=\grs$ and $\tilde g(\grs,x)= \grs x$ we obtain by Proposition 1.7 in \cite[Chapter 4]{Ethier-Kurtz:MR838085} that the processes
\begin{align}
Z_t&=Z_t^\mu=\grs_t+2\mu \int_0^t \grs_s \dd s, \nonumber\\
\tilde Z_t&=\tilde Z_t^\mu=\grs_t x_t  -\int_0^t f_{\grs_s}(x_s)\grs_s\dd s+2\mu \int_0^t \grs_s   x_s\dd s
\label{eq:martinagles}
\end{align}
are local martingales with quadratic variations $[Z]_t=4 B(t)$, where $B(t)=|\{s\in [0,t]: \grs_s \not = \grs_{-s} \}|$ denotes the number of jumps of $\grs$ up to time $t\geq 0$, and
\[
[\tilde Z]_t=4\sum_{s\in [0,t]:\ \grs_s \not = \grs_{-s}} |x_s|^2.
\]  
Moreover, the true martingale property also follows since there is a constant $C>0$ such that $|x_t|\le Ct$ for all $t>0$ due to our
assumptions on $f_{\pm1}$. Note that $(B(t))_{t\geq 0}$ is a rate $\mu$ Poisson process.

Integration of \eqref{eq:Taylor} and substitution of \eqref{eq:martinagles} gives for any $t\in[0,\infty)$
\begin{align*}
 x_t&=\int_0^t f_{\grs_s}(x_s)\dd s=\int_0^t \left(\grs_s f_1(0)+ax_s+(a_{\grs_s}-a) x_s+ R_{\grs_s}( x_s)\right) \dd s\\
&=a \int_0^t  x_s \dd s+\frac{f_1(0)}{2\mu}(Z_t-\grs_t)+\int_0^t  (a_{\grs_s}-a)x_s \dd s+ \int_0^t R_{\grs_s}( x_s) \dd s.
\end{align*}
Noting that $a_{\grs_s}-a=(\Delta a/2) \grs_s$ with $\grD a= a_1-a_{-1}$, we can  use \eqref{eq:martinagles} to rewrite this as
\begin{align*}
x_t=&a \int_0^t  x_s \dd s+\frac{f_1(0)}{2\mu}(Z_t-\grs_t)+\frac{\grD a}{4\mu}\int_0^t f_{\grs_s}(x_s)\grs_s \dd s+\frac{\grD a}{4\mu}(\tilde Z_t-\grs_t x_t) +\int_0^t R_{\grs_s}( x_s) \dd s.
\end{align*}
The variation of constants formula gives for any $t\in[0,\infty)$
\begin{multline}
 x_t=e^{at} \left(\frac{f_1(0)}{2 \mu} \int_0^t e^{-a s} \dd (Z-\grs)_s+\frac{\grD a}{4\mu}\int_0^t e^{-as} f_{\grs_s}(x_s)\grs_s\dd s
 \right.\\ \left.
 +\frac{\grD a}{4\mu}\int_0^t e^{-as}\dd(\tilde Z-\grs x)_s +  \int_0^t  e^{-a s} R_{\grs_s}( x_s) \dd s \right).
\label{chov}
\end{multline}

\begin{lem}
Suppose there is a sequence of stopping times $\theta^\mu$ with respect to the natural filtration of $\grs$ satisfying
\[
\theta^\mu\overset{P}{\rightarrow}\infty, \qquad \textrm{as } \mu\to\infty.
\]
Then as $\mu\to \infty$, the random variable
\begin{equation}
\label{eq:I-mu}
I^\mu=\frac{1}{\sqrt \mu}\int_0^{\theta^\mu} e^{-a s} \dd (Z^\mu-\grs)_s
\end{equation}
converges in distribution to $\cN(0, 2a^{-1})$.
\label{normlimit}
\end{lem}

\begin{proof} By the alternating series test,  $\int_0^{\theta^\mu} e^{-a s} \dd \grs_s$ exists and belongs to the interval $(-2,2)$. Consequently, $\mu^{-1/2}\int_0^{\theta^\mu} e^{-a s} \dd \grs_s\to 0$ as $\mu\to\infty$. Therefore, we only need to study convergence in distribution of the part with the martingale integrator.

For the rest of the proof, we follow the ideas in the proof of \cite[Lemmas 3.1, 3.2]{Bakhtin:MR2935120}. First, we define a martingale 
\begin{equation}
V_t=V^\mu_t=\int_0^{t\wedge \theta^\mu}e^{-a s} \dd Z_s
\label{eq: V-mart}
\end{equation}
with quadratic variation $[V]_t=\int_0^{t\wedge \theta^\mu}e^{-2a s} \dd [Z]_s$. Then we define a time-changed martingale $U_t=U^\mu_t=V_{g(t)}$ for $t\in[0, 2a^{-1}]$, where
\[
g(s)=-\frac{\log(1-as/2)}{2a}, \qquad s\in[0, 2a^{-1}),
\]
and $g(2a^{-1})=\infty$. 
We will prove that for any $t\in[0, 2a^{-1}]$, the quadratic variation of $U$ satisfies
\begin{equation}
\label{eq:lln-for-quad-var}
\mu^{-1} [U]_t\stackrel{P}{\to} t,\quad \mu\to\infty.
\end{equation}
Then, by  \cite[Theorem 3.1]{Bakhtin:MR2935120}, which is just a specific case of Theorem 1.4 in \cite[Chapter 7]{Ethier-Kurtz:MR838085},  $\mu^{-1/2}U_t$ converges to $\cN(0,t)$ in distribution for any $t\in[0,2a^{-1}]$ as $\mu\to \infty$. Therefore, $V^\mu_\infty = U^\mu_{2a^{-1}} \overset{d}{\rightarrow} \cN(0, 2a^{-1})$ as $\mu\to\infty$.

It remains to prove~\eqref{eq:lln-for-quad-var}. For all $t\in[0, 2a^{-1}]$,
\begin{equation}
\label{eq:quad-var-sum-over-jumps}
[U]_t=\underset{\grs(s)\not = \grs(s-)}{\sum_{s: s\leq g(t)\wedge \theta^\mu}} H(s)
\end{equation}
for $H(s)=4 e^{-2s}$.
We claim that for any non-increasing function $H(\cdot)$,
\begin{equation}
\label{eq:lln-sums-to-integrals}
\mu^{-1}\underset{\grs(s)\not = \grs(s-)}{\sum_{s: s\leq g(t)\wedge \theta^\mu}} H(s)\stackrel{P}{\to}
\int_0^{g(t)}H(s) \dd s, \quad \mu\to\infty.
\end{equation}
To prove this relation, we first note that it holds for $H(s)=\1_{[0,h]}(s)$, for any $h>0$,  
since the Law of Large Numbers implies
\begin{align*}
 \mu^{-1}\underset{\grs(s)\not = \grs(s-)}{\sum_{s: s\leq g(t)\wedge \theta^\mu}} \1_{[0,h]}(s)
=\mu^{-1} B(g(t)\wedge h \wedge \theta^\mu) 
  \overset{P}{\rightarrow} g(t)\wedge h = \int_0^{g(t)} \1_{[0,h]}(s) \dd s.
\end{align*}
Using this and approximating monotone functions with sums of indicator functions, we obtain~\eqref{eq:lln-sums-to-integrals} which, combined with \eqref{eq:quad-var-sum-over-jumps}, 
gives~\eqref{eq:lln-for-quad-var}  and completes the proof of the lemma.
\end{proof}

 For any $\gamma>0$ we can define $\theta^\mu=\inf\{t: |x_t|\geq \mu^{-\grg}\}$. Note that  for $\mu$ large enough to ensure $\mu^{-\grg}\leq r$, we have $\theta^\mu \leq\tau^\mu(r).$

\begin{lem}
The random variables
\begin{equation}
J^\mu=\int_0^{\theta^\mu} e^{-as}\dd(\tilde Z-\grs x)_s
\label{eq:J-mu}
\end{equation}
satisfy
\begin{equation}
P(|J^\mu|>\mu^{-\grd+1/2})\to 0
\label{eq:Jconv}
\end{equation} 
for any $\grd<\grg$ as $\mu\to \infty$. Consequently, the sequence $(\mu^{-1/2}J^\mu)$ converges to zero in probability as $\mu\to\infty$.
\label{Jconv}
\end{lem}

\begin{proof}
For any $t\geq 0$, we use integration by parts to write
\[
\int_0^{t} e^{-as}\dd(\grs x)_s =e^{-at}\grs_t x_t + a \int_0^t e^{-as} \grs_{s} x_{s} \dd s,
\]
\[
\left|\int_0^{\theta^\mu} e^{-as}\dd(\grs x)_s\right|\leq e^{- a \grth^\mu } |\grs_{\grth^\mu} x_{\grth^\mu}| + a \int_0^\infty e^{-as}|\grs_{s} x_{s} |\dd s \leq 2\mu^{-\grg}.
\] 
For the other term of $J$ including the martingale integrator, we can apply Chebyshev inequality followed by Proposition 6.1 in \cite[Chapter 2]{Ethier-Kurtz:MR838085}
\begin{align*}
P\left(\left| \int_0^{\theta^\mu} e^{-as}\dd\tilde Z_s \right| >\mu^{-\grd+1/2}\right)&\leq \mu^{2\grd-1} \EEE\left[ \left( \int_0^{\theta^\mu} e^{-as}\dd\tilde Z_s \right)^2\right]=\mu^{2\grd-1}\EEE\left[  \int_0^{\theta^\mu} e^{-2as}\dd[\tilde Z]_s \right].
\end{align*}
Since the quadratic variation $[\tilde Z]_{s\wedge\theta^\mu}$ is stochastically dominated by $4\mu^{-2\grg} B(s)$, by the Campbell formula in the second step
\[
 \EEE\left[  \int_0^{\theta^\mu} e^{-2as}\dd[\tilde Z]_s \right]\leq \EEE\left[  \int_0^{\infty} e^{-2as}4\mu^{-2\grg} \dd B(s) \right]  = \frac{4\mu^{1-2\grg}}{2a}, 
\]
which multiplied by $\mu^{2\grd-1}$ converges to zero.
\end{proof}

\begin{lem} If $0<\gamma<\frac{1}{2}$, then $\theta^\mu\overset{P}{\rightarrow}\infty$ as  $\mu\to\infty$.
\end{lem}
\begin{proof}
We need to check that for an arbitrary $T>0$,
$P(\grth^\mu<T)\to 0$ as $\mu\to\infty$. Recalling the definition of $V$ in \eqref{eq: V-mart} and estimating
\[
\EEE[V]_{\theta^\mu}\leq\frac{2\mu}{a}
\]
by Campbell's formula, then using the Chebyshev inequality and the martingale property of  $V^2-[V]$ (see Proposition 6.1 in \cite[Chapter 2]{Ethier-Kurtz:MR838085}), we obtain
\begin{equation}
\label{eq:estimating-V}
P\left(\frac{1}{2\mu}\left|V_{\theta^\mu}\right|>\mu^{-\grg/2-1/4}\right)\leq\frac{\EEE [V]_{\theta^\mu}}{ 4\mu^2 \mu^{-\grg-1/2}}\leq\frac{2\mu a^{-1}}{4 \mu^2 \mu^{-\grg-1/2}}\to0
\end{equation}
as $\grg<\frac{1}{2}$. By \eqref{eq:Jconv}, we also have
\begin{equation}
P\left(\frac{1}{\mu}|J^\mu|>\mu^{-\grg/2-1/4}\right)\to 0.
\label{eq:estimating-J}
\end{equation}
On the set $\{\grth^\mu<T\}$, we can use (\ref{chov}) and \eqref{eq:c} to see that if $\mu$ is large enough to guarantee $\sup_{\grs\in\{-1,1\}, x\in[-\mu^{-\grg},\mu^{-\grg}] } |f_\grs(x)|\leq 2 |f_1(0)|$, then
\begin{align*}
\mu^{-\grg}=|x_{\grth^\mu}|&=e^{a \grth^\mu} \left|\frac{f_1(0)}{2 \mu} \int_0^{\theta^\mu} e^{-a s} \dd (Z-\grs)_s+  \int_0^{\theta^\mu} e^{-a s} R_{\grs_s}(x_s) \dd s\right. \\
&\qquad+\left.\frac{\grD a}{4\mu}\int_0^{\grth^\mu} e^{-as} f_{\grs_s}(x_s)\grs_s\dd s+\frac{\grD a}{4\mu}J^\mu \right| \\
&\leq e^{aT} \left( \frac{|f_1(0)|}{2\mu}\left|V_{\theta^\mu}\right| + \frac{|f_1(0)|}{\mu}+\frac{c \mu^{-2\grg}}{2a} + \frac{|f_1(0)\grD a|}{2\mu a}+\frac{|\grD a|}{4\mu} |J^\mu| \right).
\end{align*}
Therefore, on the set $\{\grth^\mu<T\}\cap \left\{\frac{1}{2\mu}\left|V^\mu_{\theta^\mu}\right|\leq\mu^{-\grg/2-1/4}, \frac{1}{\mu}|J^\mu|\leq\mu^{-\grg/2-1/4}\right\}$,
\[
\mu^{-\grg}\leq  e^{aT} \left( |f_1(0)| \mu^{-\grg/2-1/4} + \frac{|f_1(0)|}{\mu}+\frac{c \mu^{-2\grg}}{2 a}+ \frac{|f_1(0)\grD a|}{2\mu a}+\frac{|\grD a|}{4} \mu^{-\grg/2-1/4} \right)
\]
that is impossible for large $\mu$ and $\gamma<\frac{1}{2}$. Combining this with~\eqref{eq:estimating-V} and ~\eqref{eq:estimating-J}, we complete the proof.
\end{proof}

\begin{lem} If $\frac{1}{4}<\gamma<\frac{1}{2}$, then
\[
\left( \mathrm{sgn}( x_{\theta^\mu}),\ \theta^\mu-\frac{1/2-\gamma}{a}\log\grm\right)\overset{d}{\rightarrow}\left(\mathrm{sgn}(H),\ - \frac{1}{a}\log |H|\right), \quad  \mu\to\infty,
\]
where $H$ is a random variable with  $\mathrm{Law}(H)=\cN(0, f_1(0)^2 (2a)^{-1})$.
\label{exitth}
\end{lem}

\begin{proof}
Using (\ref{chov}) at time $\grth^\mu$, 
and introducing 
\begin{align}
\label{eq:H-mu}
H^\mu=\frac{f_1(0)}{2} I^\mu  + \sqrt \mu \int_0^{\theta^\mu} e^{-a s} R_{\grs_s}(x_s) \dd s 
+\frac{\grD a}{4\sqrt\mu}\int_0^{\grth^\mu} e^{-as} f_{\grs_s}(x_s)\grs_s\dd s+\frac{\grD a}{4\sqrt\mu}J^\mu,
\end{align}
where $I^\mu$ was defined in~\eqref{eq:I-mu} and $J^\mu$ in \eqref{eq:J-mu},
we obtain 
\begin{equation}
\label{eq:expression-for-theta-mu}
\grth^\mu=\frac{1/2-\grg}{a}\log \mu - \frac{1}{a}\log \left| H^\mu \right|
\end{equation}
and
\begin{equation}
\label{eq:exit-sign}
\sgn (x_{\theta^\mu}) = \sgn H^\mu.
\end{equation}
The first term on the right-hand side of~\eqref{eq:H-mu} converges in distribution to $H$ by 
Lemma~\ref{normlimit}.  The second term converges to $0$ almost surely due to
\[
\sqrt \mu \left|\int_0^{\theta^\mu} e^{-a s} R_{\grs_s}(x_s)\dd s\right| \leq \frac{\sqrt \grm}{a}  c\ 2^{-1} \grm^{-2\grg}\to 0.
\]
Since $f_{\grs_s}(x_s)$ and $\grs_s$ are bounded, the third term on the right-hand side of~\eqref{eq:H-mu} converges to~$0$ almost surely, and the fourth term converges to zero in probability by Lemma \ref{Jconv}. Since the distribution of $H$ has no atom at $0$, the only point of discontinuity of functions
$x\mapsto \log|x|$ and $x\mapsto \sgn x$, the lemma follows now from \eqref{eq:expression-for-theta-mu} and
\eqref{eq:exit-sign}.
\end{proof}

Now, we consider the exit time from the fixed interval $[-r,r]$. We define $\grh_t=x_{\theta^\mu+t}$, $\grs'_t=\grs_{\theta^\mu+t}$ and the exit time
\[
\nu(r)=\inf\{t\geq 0: |\grh_t|\geq r\}.
\]
Recalling~\eqref{eq:F}, we define $(S^t)_{t\ge 0}$ as the flow 
associated with the ODE \[\dot z = F(z).\] Note that $(S^t)$ preserves the sign of the initial condition. For $\grd\not=0$ we introduce $t(\grd,r)$ to be the only solution of $|S^t\grd |= r$.

\begin{lem} For any $r\not= 0$,
\[
\lim_{\grd\to 0} \left( t(\grd,r)- \frac{1}{a}\log\frac{r}{\grd} \right)=K(r),
\]
where $\grd$ approaches zero from the right if $r>0$, and from the left if $r<0$,
and $K(r)$ is given in~\eqref{eq:K}. 
\label{dettime}
\end{lem}
\begin{proof}
By separation of variables,
\begin{align*}
t(\grd, r)&=\int_\grd^r \frac{\dd x}{F(x)}=\int_\grd^r \left(\frac{1}{F(x)}-\frac{1}{ax}\right) \dd x + \int_\grd^r\frac{1}{ax} \dd x\\
&=\int_\grd^r \left(\frac{1}{F(x)}-\frac{1}{ax}\right) \dd x+\frac{1}{a}\log \frac{r}{\grd}.
\end{align*}
Letting $\grd\to 0$ and using $F'(0)=a$, we complete the proof.
\end{proof}

\begin{lem}
There is $r_0>0$ such that for any $r\in (0,r_0)$
\[
\sup_{0\leq t\leq t(\grh_0, r\cdot \sgn(\grh_0))} |\grh_t-S^t \grh_0|\overset{P}{\to} 0 \quad \textrm{as } \mu\to\infty.
\]
\label{supr}
\end{lem}
\begin{proof}
Choosing $g(\grs,x)=f_{-\grs}(x)$, we obtain for $|x|<R$
\[
\cL g(\grs,x)=f_\grs(x)f'_{-\grs}(x)+2\mu \grs G(x),
\]
where $G(x)=\frac{1}{2}(f_1(x)-f_{-1}(x))>0$.
By Proposition 1.7 in \cite[Chapter 4]{Ethier-Kurtz:MR838085}, the process $Z'(t)=-g(\grs'_t,\grh_t)+\int_{0}^t \cL g(\grs'_s,\grh_s)\dd s$ is a martingale.

We will only prove
\begin{equation}
\label{eq:following determenistic closely}
\left(\sup_{0\leq t\leq t(\grh_0, r\cdot\sgn(\grh_0))} |\grh_t-S^t \grh_0|\right) \1_{\{\sgn(\grh_0)=\alpha\}}\overset{P}{\to} 0,\quad \mu\to 0,
\end{equation}
for $\alpha=1$. The case of $\alpha=-1$ is treated similarly. 

The rest of the proof follows closely the proof of \cite[Lemma 3.6]{Bakhtin:MR2935120}. We define $\grD(t)=\grh_t-S^t\grh_0$. Since for $t\geq 0$,
\begin{align*}
\grh_t &=\grh_0+\int_{0}^t F(\grh_s) \dd s + \int_{0}^t \grs_s' G(\grh_s) \dd s, \\
S^t\grh_0 &= \grh_0+\int_{0}^t F(S^t\grh_0) \dd s,
\end{align*}
we have
\[
\grD(t)=\int_{0}^t (F(\grh_s) - F(S^t\grh_0)) \dd s + \int_{0}^t \grs_s' G(\grh_s) \dd s.
\]
Using the Lipschitz constant $L(r)$ of $F$ on $[-r,r]$, we obtain on the event $\{\sgn(\grh_0)=1\}$ for any $t\leq t(\grh_0, r)=t(\mu^{-\grg}, r)$
\begin{equation}
|\grD(t\wedge\nu(r))|\leq L(r) \int_0^{t\wedge\nu(r)}|\grD(s)| \dd s + \left|\int_{0}^{t\wedge \grn(r)} \grs_s' G(\grh_s) \dd s\right|.
\label{eq:before Gronwall}
\end{equation}
Note that by the definition of $Z'$, we have $2\mu \int_{0}^t \grs_s' G(\grh_s) \dd s=Z'(t)+f_{-\grs'_t}(\grh_t)-\int_0^t f_{\grs'_s}(\grh_s)f'_{-\grs'_s}(\grh_s) \dd s$. Defining $A_1(r)=\sup_{\grs\in\{1,-1\}, z\in[-r,r]} |f_{\grs}(z)|$ and $A_2(r)=\sup_{\grs\in\{1,-1\}, z\in[-r,r]} |f_{\grs}(z) f'_{-\grs}(z)| $, we can bound
\begin{align*}
2\mu \left|\int_{0}^{t\wedge \grn(r)} \grs_s' G(\grh_s) \dd s\right| \leq
\sup_{s\leq t(|\grh_0|, r)\wedge \grn(r)} |Z'(s)|+ A_1(r)+ A_2(r)(t(\grh_0, r)\wedge \grn(r))
\end{align*}
for any $t\in [0, t(\grh_0, r)]$ on the event $\{\sgn(\grh_0)=1\}$. We claim that if $\grd<1/2$, then
\[P\left((2\mu)^{-1} \sup_{s\leq t(\grh_0, r)\wedge \grn(r)} |Z'(s)| > \mu^{-\grd}, \sgn(\grh_0)=1\right)\to 0,\quad \mu \to \infty.\]
The quadratic variation process $[Z']_t\leq 4A_1^2(r) (B(t+\grt)- B(\grt))$ is stochastically dominated by $4A_1^2(r)$ times a rate $\mu$ Poisson process. Using the Chebyshev inequality followed by Burkholder--Davis--Gundy inequality \cite[Thm. 92, Chap. 7]{DM82:MR745449}, we have with some $ \tilde C>0$
\begin{align*}
P\left( \sup_{s\leq t(\grh_0, r)\wedge \grn(r)} |Z'(s)| > 2\mu \mu^{-\grd}, \sgn(\grh_0)=1 \right) 
&\leq \tilde C \mu^{2(\grd-1)}  \EEE\left[[Z']_{t(\mu^{-\grg}, r)}\right] \\
&\leq \tilde C\mu^{2(\grd-1)}  4A_1^2(r)\, \EEE[B(t(\mu^{-\grg}, r)+\grt)- B(\grt)]\\
&= \tilde C\mu^{2(\grd-1)} \ 4A_1^2(r)\, \mu\, t(\mu^{-\grg}, r)
\end{align*}
which converges to zero by Lemma \ref{dettime} as $\mu\to\infty$, whenever $\grd<1/2$. We conclude that  for any $r>0$
\begin{align*}
P\left(\left|\int_{0}^{ t(\grh_0, r)\wedge \grn(r)} \grs_s' G(\grh_s) \dd s\right| > \mu^{-\grd}, \sgn(\grh_0)=1\right) \to 0.
\end{align*}
For $r_0>0$, on the event
\[
\left\{
\left|\int_{0}^{ t(\grh_0, r_0)\wedge \grn(r_0)} \grs_s' G(\grh_s) \dd s\right| \leq \mu^{-\grd}, \sgn(\grh_0)=1
\right\},
\]
Gronwall's inequality applied to \eqref{eq:before Gronwall} implies that there is 
$C>0$ such that for $t\leq t(\grh_0,r_0)$,
\begin{align}
|\grD(t\wedge\nu(r_0))|\leq e^{L(r_0) t(\grh_0,r_0)} \mu^{-\grd}\leq C\mu^{\grg L(r_0)/a}\mu^{-\grd}.
\label{gronwall}
\end{align}
Since $F'(0)=a$, we can choose $r_0$ so small such that $\grg L(r_0)/a<1/2$. Consequently, for some $\grd<1/2$, we have $\grr=\grd-\grg L(r_0)/a>0$ so that the right hand side in the bound above converges to 0. For any $r \in (0,r_0)$, we conclude using (\ref{gronwall}) at $t=\nu(r_0)$ that $P(\grn(r_0) < t(\grh_0, r), \sgn(\grh_0)=1) \to 0$. Now, using (\ref{gronwall}) at 
$s\leq t(\grh_0, r)<t(\mu^{-\grg}, r_0)$,  we have
\[
P\left(\grn(r_0) \geq t(\grh_0, r), \sup_{s\leq t(\grh_0, r)} |\grD(s)| > C \mu^{-\grr},  \sgn(\grh_0)=1  \right)\to 0,
\]
which implies \eqref{eq:following determenistic closely} for $\alpha=1$.
\end{proof}

Combining Lemmas \ref{exitth}, \ref{dettime} and \ref{supr}, we conclude

\begin{lem} For any $r\in(0, r_0)$
\[
\left( x_{\grt(r)}, \grt(r)- \frac{1}{2a}\log\mu  \right) \overset{d}{\to} \left(r\cdot \mathrm{sgn} (H), - \frac{1}{a}\log |H| + \frac{\log r}{a}+ K(r \cdot \sgn H)\right) \quad \mathrm{as}\ \mu\to\infty.
\]
\label{exitr}
\end{lem}

Now let us consider the exit time from an arbitrary interval.

\begin{lem}
For $r\in(0, r_0)$, define $\grh_t=x_{\grt(r)+t}$. Then, for any $T>0$,
\[
\sup_{0\leq t \leq T} |\grh_t- S^t \grh_0| \overset{P}{\to} 0 \quad \mathrm{as}\ \mu\to\infty.
\]
\label{supR}
\end{lem}
The proof follows the same steps as the proof of Lemma \ref{supr}, however easier as the time horizon is fixed this time. Our main result follows from Lemmas \ref{exitr} and \ref{supR}.

\bibliographystyle{alpha}
\bibliography{}

\end{document}